\documentclass{amsart}

\usepackage[nobysame,initials,alphabetic]{amsrefs}
\usepackage{xcolor}

\usepackage{graphicx}

\usepackage{helvet}         % selects Helvetica as sans-serif font
\usepackage{courier}        % selects Courier as typewriter font
\usepackage{type1cm}        % activate if the above 3 fonts are
\usepackage[bottom]{footmisc}% places footnotes at page bottom
\usepackage{epstopdf}
\usepackage{geometry}
\usepackage{tikz}
\usepackage{subfig}
\usepackage{makecell}
\usepackage{multirow}

\allowdisplaybreaks

\newcommand{\vol}{\text{Vol}}
\newcommand{\sign}{\text{sign}}

\usepackage{nth}
\usepackage{todonotes}

\newlength\savedwidth

\newcommand\thickhline{\noalign{\global\savedwidth\arrayrulewidth\global\arrayrulewidth 2pt}%
\hline
\noalign{\global\arrayrulewidth\savedwidth}}

\newcommand{\taylor}[2]{\left(#1\right)_{#2}}
\newcommand{\deriv}[3]{ {#3}^{({#1})}
}
\newcommand{\tmax}{1.1}
\newcommand{\xmax}{1.123841}
\newcommand{\tintv}{\ensuremath{[0, \tmax]}}
\newcommand{\xintv}{\ensuremath{[0, x(\tmax)]}}
\newcommand{\goalOrder}{7}
\newcommand{\resultBound}{0.50344}
\newcommand{\integers}{\mathbb{Z}}
\newcommand{\reals}{\mathbb{R}}
\newcommand{\scie}{\mathrm{e}\!-\!}

\newtheorem{lemma}{Lemma}
\newtheorem{theorem}{Theorem}
\newtheorem{corollary}{Corollary}

\title{Central diagonal sections of the $n$-cube}

\author{F. A. Bartha$^1$}
\address{Department of Applied and Numerical Mathematics, University of Szeged, Aradi v\'ertan\'uk tere 1, 6720 Szeged, Hungary}
\thanks{$^1$Supported by NKFIH KKP 129877 and EFOP-3.6.2-16-2017-00015 grants and by the grant TUDFO/47138-1/2019-ITM of the
	Ministry for Innovation and Technology, Hungary.}
\email{barfer@math.u-szeged.hu}

\author{F. Fodor$^2$}
\address{Department of Geometry, University of Szeged, Aradi v\'ertan\'uk tere 1, 6720 Szeged, Hungary}
\thanks{$^2$Supported by grant TUDFO/47138-1/2019-ITM of the Ministry for Innovation and Technology, Hungary, and by Hungarian National 
	Research, Development and Innovation Office NKFIH grant KF129630.}
\email{fodorf@math.u-szeged.hu}

\author{B. Gonzalez Merino$^3$}
\address{Departamento de Did\'actica de la Matem\'atica, Facultad de Educaci\'on, Universidad de Murcia, 30100-Murcia, Spain}
\thanks{$^3$ This research is a result of the activity developed within the framework of the Programme in Support
of Excellence Groups of the Regi\'on de Murcia, Spain, by Fundaci\'on S\'eneca, Science and Technology
Agency of the Regi\'on de Murcia. Partially supported by Fundaci\'on S\'eneca project 19901/GERM/15,
Spain, and by MICINN Project PGC2018-094215-B-I00 Spain.}
\email{bgmerino@um.es}

\keywords{Cube, sections, volume}

\date{\today}

\begin{document}
\maketitle

\begin{abstract}
We prove that the volume of central hyperplane sections of a unit cube in $\mathbb{R}^n$ orthogonal to a diameter of the cube is a strictly monotonically increasing function of the dimension for $n\geq 3$.
Our argument uses an integral formula that goes back to P\'olya \cite{P} (see also \cite{H} and \cite{B86}) for the volume of central sections of the cube, and Laplace's method to estimate the asymptotic behaviour of the integral. First we show that monotonicity holds starting from some specific $n_0$. Then, using interval arithmetic (IA) and automatic differentiation (AD), we compute
an explicit bound for $n_0$, and check the remaining cases between $3$ and $n_0$ by direct computation.
\end{abstract}

\section{Introduction}
Let $C^n=[-\frac 12, \frac 12]^n$ be the unit cube in $\mathbb{R}^n$, and for $u\in\mathbb{R}^n$ let $H(u)=u^\perp$, the hyperplane through $o$ orthogonal to $u$. We are interested in determining $\vol_{n-1}(C^n\cap H(u_0))$ in the special case when $u_0=(1,\ldots, 1)\in \mathbb{R}^n$ is parallel to a main diagonal of $C^n$. 

Hensley \cite{H} described a probabilistic argument, whose origin he attributed to Selberg, proving that $\vol_{n-1}(C^n\cap H(u_0))\to\sqrt{6/\pi}$ as $n\to\infty$, and he conjectured that $\max_u \vol_{n-1}(C^n\cap H(u))\leq \sqrt 2$. This conjecture was proved by Ball \cite{B86}, who proved a integral formula for the volume of sections that goes back to P\'olya \cite{P}, which, when specialized to the case of $H(u_0)$, is the following:
\begin{equation}\label{ball-formula}
I(n):=\vol_{n-1}(C^n\cap H(u_0))=\frac{2\sqrt n}{\pi}\int_0^{+\infty}\left (\frac{\sin t}{t}\right )^n dt.
\end{equation}
It is an interesting fact that the maximum volume hyperplane section of the cube occurs when the hyperplane is orthogonal to $u=(1,1,0,\ldots, 0)$, 
%that is, the hyperplane is orthogonal to a $2$-dimensional face diagonal of $C^n$, 
and not for hyperplanes orthogonal to the main diagonal. The limit $\sqrt{6/\pi}$ for the main diagonal is slightly less than $\sqrt 2$. 

%Medhurst and Roberts \cite{MR65} were probably the first ones doing a systematic study of the integral \eqref{ball-formula}. 
It is known that the integral \eqref{ball-formula} can be evaluated explicitly as%asFrank and Riede \cite{FR} evaluated (the general form of)  \eqref{ball-formula}, cf. \cite[Theorem~2]{FR}. If one specializes their formula for central diagonal sections, then one obtains that 
\begin{equation}\label{Frank-Riede}
\vol_{n-1}(C^n\cap H(u_0))=\frac{\sqrt n}{2^{n}(n-1)!}\sum_{i=0}^{n}(-1)^i{n\choose i}(n-2i)^{n-1}\sign (n-2i),
\end{equation}
see Goddard \cite{Go45}, Grimsey \cite{Gr45}, Butler \cite{Bu60}, and Frank and Riede \cite{FR}.
Numerical computations with \eqref{Frank-Riede} suggest that $\vol_{n-1}(C^n\cap H(u_0))$ is a strictly monotonically increasing function of $n$ while it tends to the limit $\sqrt{6/\pi}$ as $n\to\infty$. However, \eqref{Frank-Riede} does not seem to lend itself as a tool for proving this monotone property. 

\begin{figure}[h]
\centerline{\includegraphics[scale=0.7]{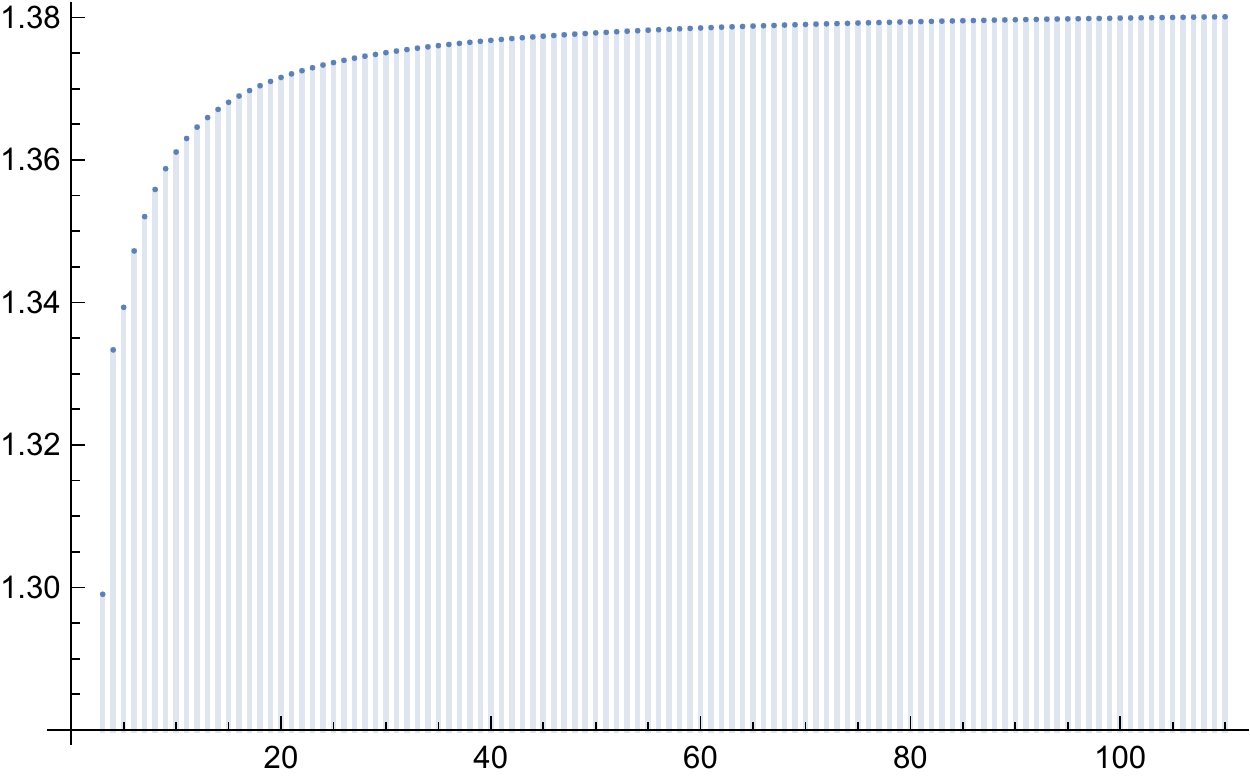}}
\caption{$\vol_{n-1} (C^n\cap H(u_0))$ for $3\leq n\leq 110$ plotted by {\em Mathematica}.\label{figure}}
\end{figure}

Recently, K\"onig and Koldobsky proved that, in fact, $\vol_{n-1}(C^n\cap H)\leq \sqrt{6/\pi}$ for all $n\geq 2$, see \cite[Prop. 6(a)]{KK19}. We also point out the recent result of Aliev \cite{Aliev2} (see also \cite{Aliev}) about hyperplane sections of the cube, in which he proves that
\begin{equation}\label{Aliev}
\frac{\sqrt n}{\sqrt{n+1}}\leq \frac{I(n+1)}{I(n)}
\end{equation}
which is slightly less than the monotonicity of $\vol_{n-1}(C^n\cap H(u_0))$.

For a more detailed overview of the currently known information on sections of the cube and for further references, see, for example, the books of Berger \cite{Berger} and Zong \cite{Z}, and the papers by Ball \cite{B86, B89},  K\"onig, Koldobsky \cite{KK19} and Ivanov, Tsiutsiurupa \cite{IT}.

Our main result is the following.
\begin{theorem}\label{main}
The volume $\vol_{n-1}(C^n\cap H(u_0))$ is a strictly monotonically increasing function of $n$ for all $n \geq 3$.
%$n\geq n_0$. 
\end{theorem}
%We note that it will become clear from our subsequent proof of Theorem~\ref{main} that getting an explicit value for $n_0$ seems to be a purely numerical task. If one has a bound for $n_0$, then one can verify monotonicity by directly calculating the section volumes, say, with the formula of Frank and Riede. 
Theorem \ref{main} directly yields the following corollary (which has already been proved by K\"onig and Koldobsky \cite{KK19}), and slightly improves the estimate \eqref{Aliev} of Aliev mentioned above.
\begin{corollary}
	For any integer $n\geq 2$, it holds that  
	$$\vol_{n-1}(C^n\cap H(u_0))< \sqrt{\frac{6}{\pi}} .$$
	and this upper bound is best possible.
\end{corollary}

The rest of the paper is organized as follows. In Section~\ref{sec:monotone} we use Laplace's method to study the behaviour of the integral \eqref{ball-formula}, and prove the existence of an integer $n_0$ with the property that $\vol_{n-1}(C^n\cap H(u_0))$ is an increasing sequence for all $n\geq n_0$.
In the Appendix, using interval arithmetic, automatic differentiation, and some analytical arguments, we provide rigorous numerical estimates, which we use in 
Section~\ref{sec:final} to obtain an explicit upper bound on $n_0$. Finally, we check monotonicity for $3\leq n\leq n_0$ by calculating the value of of $\vol_{n-1}(C^n\cap H(u_0))$ using \eqref{Frank-Riede}, thus concluding the proof of Theorem \ref{main}.

\section{Proof of the monotonicity for large $n$}\label{sec:monotone}
In this section, we prove the following statement which is the most important ingredient of the proof of Theorem \ref{main}.
\begin{theorem}\label{main-2}
There exists an integer $n_0$ such that $\vol_{n-1}(C^n\cap H(u_0))$ is a strictly monotonically increasing function of $n$ for all $n\geq n_0$. 
\end{theorem}
\begin{proof}
We are going to examine the behaviour of the integral: %(this is the $(n-1)$-dimensional volume of the section):
$$I(n)=\frac{2\sqrt n}{\pi}\int_0^{+\infty}\left (\frac{\sin t}{t}\right )^n dt, \quad n\geq 3.$$
We wish to prove that there exists an $n_0$ such that $I(n)$ is strictly monotonically increasing for all $n\geq n_0$.

We start the argument by restricting the domain of integration to a finite interval that contains most of the integral as $n\to\infty$.
If $a$ fixed, with $1<a<\pi/2$, then for $n\geq 2$ it holds that
$$\frac{2\sqrt n}{\pi}\int_a^{+\infty}\left |\frac{\sin t}{t}\right |^n dt<\frac{2\sqrt n}{\pi}\int_a^{+\infty}t^{-n}dt=\frac{2\sqrt n}{\pi} \frac{a^{-n+1}}{n-1}<2 a^{-n}=:e_1(n).$$

Note that the function $e_1(n)$ tends to $0$ exponentially fast as $n\to +\infty$. Let $a$ be fixed, say, $a=1.1$, 
and define 
\begin{equation}\label{I_a}
I_a(n):=\frac{2\sqrt n}{\pi}\int_0^{a}\left (\frac{\sin t}{t}\right )^n dt, \quad \text{ for }n\geq 3.
\end{equation}

Then
$$|I(n)-I_a(n)|<e_1(n)\quad  \text{ for } n\geq 3.$$

We will use Laplace's method to study the behaviour of $I_a(n)$.
Let us make the following change of variables
$$\frac{\sin t}{t}=e^{-x^2/6}, \text{ thus } x=\sqrt{-6\log \frac{\sin t}{t}},$$
where we define the value of $\sin t/t$ to be $1$ at $t=0$. 
Therefore, $x(t)$ is analytic in the interval $[0,a]$. 
Note that $x(0)=0$, and $x'(t)>0$ for all $t\in [0,t]$. Thus, $x(t)$ maps $[0,a]$ bijectively onto $[0,x(a)]$, and so it has an inverse $t(x):[0,x(a)]\to [0,a]$, which is also analytic in $[0, x(a)]$ by the Lagrange Inversion Theorem because $x'(t)\neq 0$ for $t\in [0,a]$.  
In our case, 
$1.07<x(a)=x(1.1)=1.07768<1.08$.

The Taylor series of $x(t)$ around $t=0$ begins with the terms 
$$x=t+\frac{t^3}{60}+\frac{139t^5}{151200}+\frac{83t^7}{1296000}%+\frac{99049t^9}{20118067200}
+\ldots.$$

%On can easily check by simple differentiation that $x'(t)\neq 0$ for all $t\in [0,a]$ 
%\textcolor{blue}{and $x'_+(0)=\lim_{t\rightarrow 0+}x(t)/t=1$,} and thus, by the Lagrange Inversion Theorem, the inverse function $t(x)$ is analytic at each point of $[0, a]$.

We can get the first few terms of the the Taylor series expansion of $t=t(x)$ around $x=0$ by inverting the Taylor series of $x(t)$ at $t=0$ as follows:
$$t(x)=x-\frac{x^3}{60}-\frac{13x^5}{151200}+\frac{x^7}{33600}%+\frac{17597x^9}{100590336000}
+\ldots.$$
Then
$$t'(x)=1-\frac{x^2}{20}-\frac{13x^4}{30240}+R_6(x)$$
is the order $5$ Taylor polynomial of $t'(x)$ around $x=0$ (observe that the degree $5$ term is zero), and for the Lagrange remainder term $R_6(x)$, it holds that
$$R_6(x)=\frac{t^{(7)}(\xi)}{6!}x^6$$
for some $\xi\in (0,x)$ (depending on $x$). 
Since $t(x)$ is analytic in $[0,x(a)]$, in particular 
the seventh derivative of $t(x)$ is analytic too, and thus it is a continuous function. Then the Extreme Value Theorem yields that $t^{(7)}$ attains its maximum in $[0,x(a)]$, and thus $|t^{(7)}(x)|\leq R$, for some $R>0$ and every $x\in[0,x(a)]$. %Indeed, plotting $t^{(7})$ in \textit{Mathematica} shows evidence that it is an increasing function in $[0,x(a)]$.
%\textcolor{blue}{Is it needed more detail than what is written before?}
Then we can use the following estimate on $x\in [0, x(a)]$:
\begin{equation}\label{remainder}
|R_6(x)|\leq \frac{R}{6!}x^6,
\end{equation}

Therefore, after the change of variables, we need to evaluate
\begin{align*}
I_a(n)&=\frac{2\sqrt n}{\pi}\int_0^{x(a)}e^{-nx^2/6} t'(x) dx\\
&=\frac{2\sqrt n}{\pi}\int_0^{x(a)}e^{-nx^2/6}\left (1-\frac{x^2}{20}
-\frac{13x^4}{30240}%+\frac{x^6}{48000}
%+\frac{17597x^9}{11176704000}
+R_{6}(x)\right )dx\\
&=\frac{2\sqrt n}{\pi}\int_0^{x(a)}e^{-nx^2/6}\left (1-\frac{x^2}{20}
-\frac{13x^4}{30240}%+\frac{x^6}{48000}
%+\frac{17597x^9}{11176704000}
\right )dx\\
&+\frac{2\sqrt n}{\pi}\int_0^{x(a)}e^{-nx^2/6} R_{6}(x) dx.
\end{align*}

In order to calculate the above integrals we will use the central moments of the normal distribution:
If $y=\frac{1}{\sqrt{2\pi\sigma^2}}e^{-\frac{(x-\mu)^2}{2\sigma^2}}$, then for an integer $p\geq 0$ it holds that
\begin{equation}\label{int-moments}
{\mathbb E}[y^p]=
\begin{cases}
0, &\text{if $p$ is odd},\\
\sigma^p(p-1)!!, &\text{ if $p$ is even}.
\end{cases}
\end{equation}
%Furthermore, if $q=p+\frac 12$ (see \cite[(15)]{W} \textcolor{red}{here we should have a better reference}), then
 %for a real number $p>-1$ it holds that (\textcolor{red}{this is only for half-integers $p=n/2$ and integrating from $0$ to $+\infty$, in fact, this is the real part of the moment, we need a reference here!!! and also some better notation})
%\begin{equation}\label{halfint-moments}
%\int_0^{+\infty}y\cdot x^{q} dx=\mathrm{Re}\, ({\mathbb E}[y^q])=\sigma^q2^{q/2-1}\frac{\Gamma(\frac{q+1}{2})}{\sqrt{\pi}},
%\end{equation}
%where $\Gamma(\cdot)$ denotes Euler's gamma function.

In our case $\sigma^2=3/n$. Thus, using %\eqref{halfint-moments} and 
\eqref{int-moments} and \eqref{remainder}, we get that
\begin{align*}
\frac{2\sqrt n}{\pi}\int_0^{x(a)}e^{-nx^2/6} |R_{6}(x)| dx&\leq \frac{2R\sqrt n}{\pi 6!}\int_0^{x(a)}e^{-nx^2/6}x^6 dx\\
&< \frac{2R\sqrt n}{\pi 6!}\int_0^{+\infty}e^{-nx^2/6}x^6 dx\\
&=\frac{2R\sqrt n}{\pi 6!}\frac{3^3}{n^3}5!!\\
&=\frac{9 R}{8\pi}\frac{1}{n^{5/2}}\\
&<\frac R2\frac{1}{n^{5/2}}=:e_2(n).
\end{align*}

%\begin{align*}
%\frac{2\sqrt n}{\pi}\int_0^{x(a)}e^{-nx^2/6} |R_{6}(x)| dx&\leq \frac{2R\sqrt n}{\pi}\int_0^{x(a)}e^{-nx^2/6}x^{11/2} dx\\
%&< \frac{2R\sqrt n}{\pi}\int_0^{+\infty}e^{-nx^2/6}x^{11/2} dx\\
%&=2R\sqrt{6}\frac{3^{11/2}}{\pi}\frac{2^{11/4-1}\Gamma(\frac{13}{4})}{n^{11/4}}\\
%&=\frac{2^{13/4}\cdot 3^{6}\cdot\Gamma(\frac{13}{4})R}{\pi }\frac{1}{n^{11/4}}\\
%&\leq \frac{5630\cdot R}{n^{11/4}}=:e_2(n).
%&=\frac{C}{n^{11/4}}=:e_2(n).
%\end{align*}

Notice also that
\begin{align*}
&\frac{2\sqrt n}{\pi}\int_{0}^{+\infty}e^{-nx^2/6}\left (1-\frac{x^2}{20}
-\frac{13x^4}{30240}%+\frac{x^6}{48000}
%+\frac{17597x^9}{11176704000}
\right )dx\\
&=\sqrt{\frac{3\pi}{2}}\frac{2\sqrt n}{\pi}\left (\frac{1}{n^{1/2}}-\frac{3}{20n^{3/2}}
-\frac{13}{1120n^{5/2}}%+ \frac{27}{3200n^{7/2}}+\frac{52791}{3942400n^{9/2}}
\right )\\
&=\sqrt{\frac 6\pi}\left (1-\frac{3}{20n}
-\frac{13}{1120n^2}%+ \frac{27}{3200n^3}+\frac{52791}{3942400n^4}
\right ).
\end{align*}

The complementary error function is defined as
$$\textrm{erfc}(x):=2\frac{1}{\sqrt{\pi}}\int_{x}^{+\infty} e^{-\tau^2} d\tau.$$
It is known that $\mathrm{erfc}(x)\leq e^{-x^2}$ for $x\geq0$. Then, taking into accout that $x(a)>1.07$, we obtain
\begin{align*}
\frac{2\sqrt n}{\pi}&\left |\int_{x(a)}^{+\infty}e^{-nx^2/6}\left(1-\frac{x^2}{20}
-\frac{13x^4}{30240}\right)dx\right |\\
&\leq \frac{2\sqrt n}{\pi}\int_{x(a)}^{+\infty}e^{-nx^2/6}\left|1-\frac{x^2}{20}
-\frac{13x^4}{30240}\right|dx\\
&\leq \frac{2\sqrt n}{\pi} \int_{x(a)}^{+\infty}e^{-nx^2/6}\left(1+\frac{x^2}{20}
+\frac{13x^4}{30240}\right)dx\\
&<\frac{2\sqrt n}{\pi}\int_{1}^{+\infty}e^{-nx^2/6}\left(1+\frac{x^2}{20}+\frac{13x^4}{30240}\right)dx\\
&=\sqrt{\frac{6}{\pi}}\mathrm{erfc}(\sqrt{n/6})\left (\frac{13+168n+1120 n^2}{1120 n^{2}}\right )+2e^{-n/6}\sqrt n\frac{117+1525n}{10080 \pi n^{2}}
%&<\mathrm{erfc}(0.44\sqrt{n})\left(\frac{2.2}{\sqrt{n}}-\frac{0.3}{n^{3/2}}-\frac{0.02}{n^{5/2}}\right)
%+e^{-0.19n}\left(-\frac{0.1}{n}-\frac{0.01}{n^2}\right)
\\
&<5e^{-n/6}=:e_3(n).
\end{align*}

Now, using the monotonicity of $e_1(n)$, we obtain that
\begin{multline*}
I(n+1)-I(n)\geq (I_a(n+1)-e_1(n+1))-(I_a(n)+e_1(n))\geq I_a(n+1)-I_a(n)-2e_1(n).
\end{multline*}
Furthermore,
$$I_a(n+1)\geq \sqrt{\frac 6\pi}\left (1-\frac{3}{20(n+1)}-\frac{13}{1120(n+1)^2}\right )-e_2(n+1)-e_3(n+1),$$
and
$$
I_a(n)\leq \sqrt{\frac 6\pi}\left (1-\frac{3}{20n}-\frac{13}{1120 n^2}\right)+e_2(n)+e_3(n).
$$
Therefore
\begin{align}
I(n+1) -I(n) & \geq \sqrt{\frac 6\pi}\left (\frac{3}{20n}-\frac{3}{20(n+1)}+\frac{13}{1120n^2}-\frac{13}{1120(n+1)^2} \right)\notag\\
&-2e_1(n)-e_2(n)-e_2(n+1)-e_3(n)-e_3(n+1)\notag\\
&=\sqrt{\frac 6\pi}\left (\frac{3}{20n(n+1)}+\frac{13(2n+1)}{1120n^2(n+1)^2}\right )\notag\\
&-
4a^{-n}-(e_2(n)+e_2(n+1)+e_3(n)+e_3(n+1))\notag\\
&>\sqrt{\frac 6\pi}\left (\frac{3}{20n(n+1)}
\right )
-4a^{-n}-
2e_2(n)-2e_3(n)\notag\\
&>\sqrt{\frac 6\pi}\left (\frac{3}{20n(n+1)}
\right )-
4a^{-n}
-\frac{R}{n^{5/2}}
-10e^{-n/6}
\notag\\
&\geq \sqrt{\frac 6\pi}\left (\frac{3}{20n(n+1)}
\right )-4\cdot  1.1^{-n}
-\frac{R}{n^{5/2}}
-10e^{-n/6}.
\label{estimate}
\end{align}
Clearly, there exists an $n_0$, such that for all $n\geq n_0$ the expression 
\eqref{estimate} is strictly positive. Thus, $\vol_{n-1}(C^n\cap H(u_0))$ is strictly monotonically increasing for $n\geq n_0$. 

Thus, we have finished the proof of Theorem~\ref{main-2}. 
\end{proof}

\noindent {\em Remark.}\, 
Figure~\ref{figure} suggests that $\mathrm{Vol}_{n-1}(C_n\cap H(u_0))$ is not only a monotonically increasing sequence but also concave, i.e., $2I(n+1)\geq I(n)+I(n+2)$ for $n\geq 3$. We note, without giving the details, that with a similar argument as in the proof of Theorem~\ref{main-2}, but using more terms of the Taylor expansion of $t(x)$, one can also show that
\begin{align*}
2 & I(n+1)- I(n)-I(n+2) \\ 
& \geq 2I_a(n+1)-I_a(n)-I_a(n+2) - \xi_1 e_1(n)-\xi_2 e_2(n)-\xi_3 e_3(n) \\
& \geq \frac{3\sqrt{3}}{5\sqrt{2\pi}}\frac{1}{n(n+1)(n+2)}+O(n^{-4})
-\xi_1 e_1(n)-\xi_2 e_2(n)-\xi_3 e_3(n),
\end{align*}
for some $\xi_i>0$, $i=1,2,3$. If we take into account sufficiently many terms of the Taylor series of $t(x)$, then we can guarantee that each error term is of smaller order than $n^{-3}$, and thus
there exists a number $n_1$ such that the sequence $I(n)$ is concave for all $n\geq n_1$.
% for some integer $n_1$, if each $e_i(n)$ would be of order smaller than $n^{-3}$. This is clearly fulfilled by $e_1$ and $e_3$. However, $e_2(n)=O(n^{-5/2})$. Therefore, the result holds and can be justified by considering more terms in the Taylor expansion of $x(t)$ (maybe up to the seventh term).

%\section{Estimating $R$}\label{sec:R-estimate}

\section{Proof of Theorem~\ref{main}}\label{sec:final}
In order to prove Theorem~\ref{main}, we need an explicit upper bound on the critical number $n_0$. Using a combination of interval arithmetic, automatic differentiation, and some analytic methods, we can obtain a rigorous upper estimate for the seventh derivative $|t^{(7)}(x)|$ in $x\in [0, x(a)]$. We provide the details of this argument in the Appendix. Here, we only quote the following upper bound (see Theorem \ref{app:thm:main} part \eqref{app:eq:result}):
\begin{equation}\label{R}
R\leq 0.50344.
\end{equation}
Now, substituting the estimate \eqref{R} in inequality \eqref{estimate}, we get that $n_0< 145$. Then, we can calculate the values of $I(n+1)-I(n)$ using \eqref{Frank-Riede} to the required accuracy, and verify monotonicity for all $3\leq n\leq 145$, see Figure~\ref{fig:graphs} below.

\begin{figure}[h]
	\includegraphics[scale=0.4]{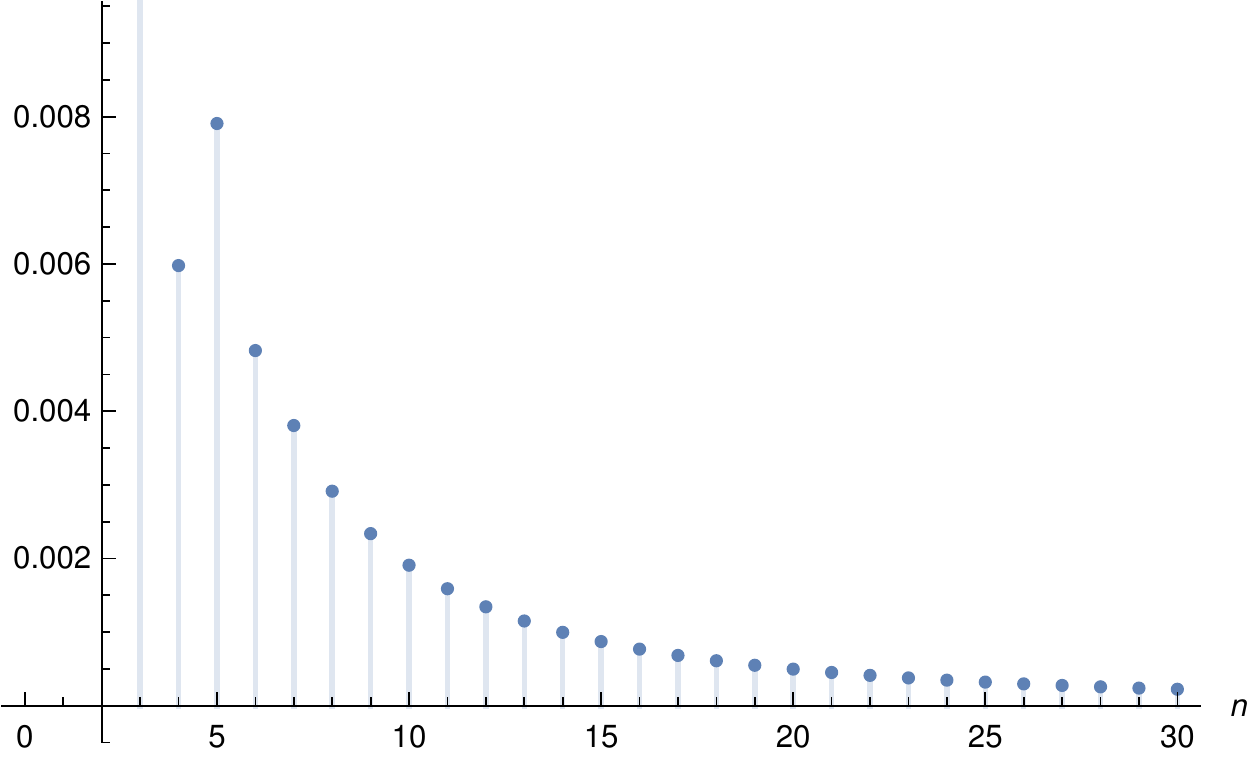}
	\includegraphics[scale=0.4]{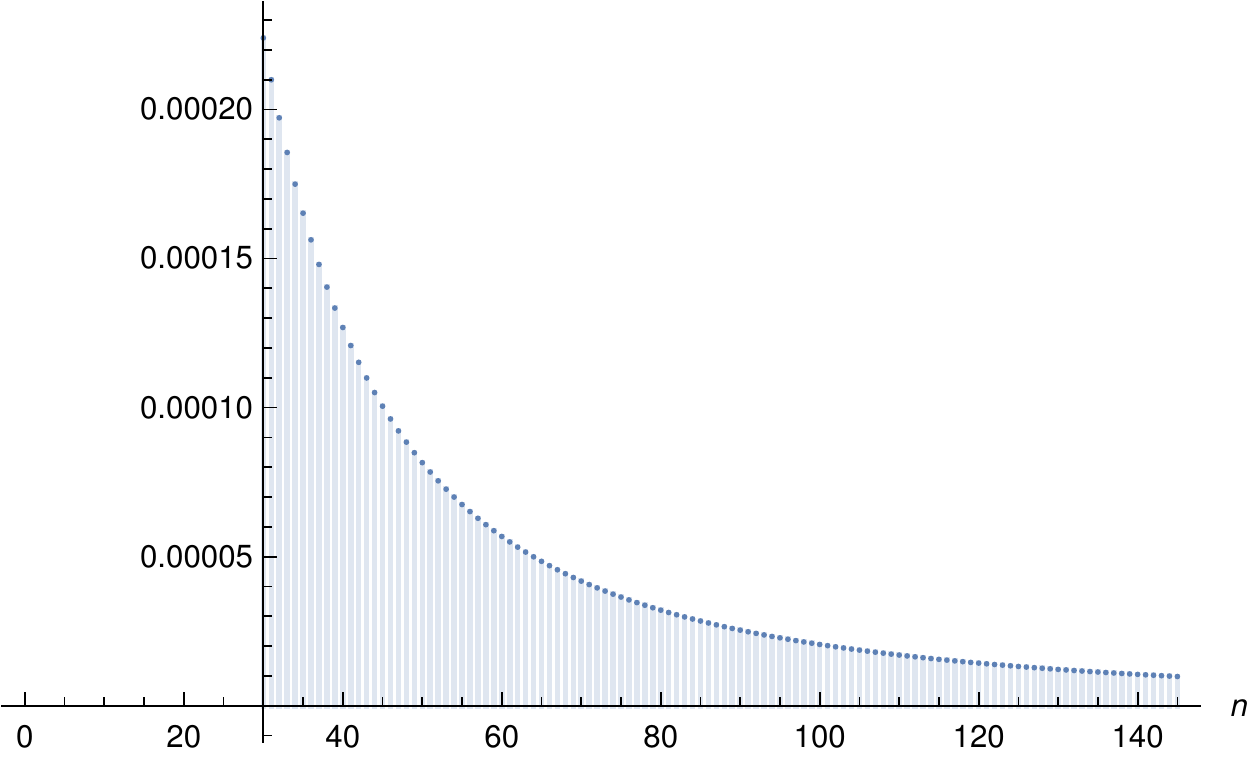}
\caption{$I(n+1)-I(n)$ for $3\leq n\leq 145$ plotted by {\em Mathematica}\label{fig:graphs}}
\end{figure}

%A direct consequence of Theorem \ref{} (write later on from Section \ref{sec:R-estimate}) is that $R$ in \eqref{estimate} can be chosen so that $R=1.25$ (I have seen that Ferenc B has improved it further!). It is very simple to check that for such value, \eqref{estimate} is strictly larger than $0$ for $n\geq 145$ (change correspondingly as above changes too! Should we show a table here as you did in the .nb file?). In order to conclude the proof of Theorem \ref{main}, we use the formula of Frank and Riede \eqref{Frank-Riede} with the aid of \emph{Mathematica} to show that $I(n+1)-I(n)>0$ for every $3\leq n\leq 145$ (change this 145 later on!). 

%\textcolor{red}{We note that we do not know a \emph{good} reason for the exception that the sequence decreases from dimension $2$ to $3$. DO WE NEED THIS?}

\noindent{\em Remark.}\, Using the same ideas as above, one could show the concavity of $I(n)$ for $n\geq 3$.

\section{Acknowledgements} We would like to thank Juan Arias de Reyna for helpful discussions and suggestions.

\vskip 4em

\section*{Appendix}

%\section{Introduction}
%\label{app:sec:intro}

Consider the function
\begin{equation}
    \label{app:eq:xt}
    x(t) = \sqrt{-6 \log \left(\frac{\sin t}{t} \right)}, 
\end{equation}
where $t \in \tintv$. The fraction $\tfrac{\sin t}{t}$ is 
understood to be augmented with its limit at $t = 0$ 
that is $\tfrac{\sin 0}{0} = 1$. 
Then, the function $x(t)$ is analytic. 
\begin{theorem}
\label{app:thm:main}
The following holds true. 
\begin{enumerate}
    \item $x(t)$ is strictly increasing on $\tintv$ 
    and 
    \begin{equation*}
    \label{app:eq:xmax}
        x(t) \leq \xmax \qquad \mbox{for } t \in \tintv. 
    \end{equation*}
    \item $x(t)$ is invertible 
        with inverse $t(x)$, where $x \in \xintv$. 
    \item The \nth{\goalOrder} derivative of $t(x)$ attains the upper bound
    \begin{equation*}
        \label{app:eq:result}
        \left|\deriv{\goalOrder}{x}t (x)\right| \leq \resultBound \qquad 
            \mbox{for } x \in \xintv. 
    \end{equation*}
\end{enumerate}
\end{theorem}
The monotonicity stated in (1) is trivial, hence, 
one just needs to establish the containment $x(\tmax) \in \xintv$. 
Note that (2) is a consequence of (1), thus, in the following 
we will deal with evaluating $x(t)$ and proving (3). 

There are numerous computational steps involved. 
In order to obtain rigorous results, we have based 
our computations on two techniques, namely, 
{\em interval arithmetic (IA)} and 
{\em automatic differentiation (AD)} that 
are capable of providing mathematically sound bounds 
for functions and their derivatives alike. 
Besides the technical hurdle, severe difficulties arise at 
the left endpoint $t = 0$ as, when computing 
the derivatives of $x(t)$, we need to 
differentiate both $\sqrt{\cdot}$ and 
$\tfrac{\sin t}{t}$ at zero.
It was tempting to use Taylor models, 
an advanced combination of these two, 
however that could still not handle 
the aforementioned left endpoint directly, hence, we 
chose to stick with the straightforward application 
of the two techniques and used the CAPD package 
\cite{app:capd}. For a comprehensive 
overview of these topics we refer to 
\cite{app:rigref, app:rigref2, app:rigref3}. 

We emphasize that the major goal of Theorem~\ref{app:thm:main} 
is providing the given bounds, hence, we 
made little effort to obtain tighter results and were 
performing sub--optimal computations knowingly, 
in order to decrease the implementation burden. 

The key step to overcome the difficulties at $t = 0$ is to 
rephrase \eqref{app:eq:xt} as 
\begin{equation}
\label{app:eq:xt-scheme}
\begin{split}
    x(t) &= t \sqrt{h(t)}, \\
    h(t) &= (g \circ F_2)(t) \cdot (- 6 F(t)), \\
    g(t) &= \frac{\log(1 + t)}{t}, \\
    F_2(t) &= t^2 F(t), \qquad \qquad \mbox{and} \\
    F(t) &= \frac{\tfrac{\sin(t)}{t} - 1}{t^2}. 
\end{split}
\end{equation}
Section~\ref{app:sec:functions} details the 
considerations used for dealing with 
the functions appearing in \eqref{app:eq:xt-scheme}. 
In particular, 
Sections~\ref{app:sec:sintpt} and \ref{app:sec:sintpt-taylor-model} 
handle the functions 
$\tfrac{\sin t}{t}$ and $F(t)$; a computational 
scheme for their derivatives is provided. 
Then, we turn our attention to 
$\tfrac{\log(1 + t)}{t}$ and derive analogous results 
in Section~\ref{app:sec:logtpt}. The square root is discussed in 
Section~\ref{app:sec:sqrt}. Then, in Section~\ref{app:sec:chain}, 
we present a pure formula 
for the higher order chain--rule used to compose 
$g(t)$ and $F_2(t)$. Section~\ref{app:sec:xt} contains the results 
for $x(t)$ and its derivatives, in particular, the proof of 
the remaining part of (1) in Theorem~\ref{app:thm:main}. 
Section~\ref{app:sec:tx} deals with $t(x)$ by giving a general 
inversion procedure in Section~\ref{app:sec:inverse} and the 
final proof in Section~\ref{app:sec:proof}.

The codes performing the rigorous computational procedure 
described in this manuscript, 
together with the produced outputs, 
are publicly available at \cite{app:code}. 

%%%%%%%%%%%%%%%%%%%%%%%%%%%%%%%%%%%%%%%%%%%

\section{Bounding functions and their derivatives}
\label{app:sec:functions}
First, we will closely analyze some Taylor expansions centered 
at $t_0 = 0$ and derive bounds for Taylor coefficients of the 
very same functions expanded around another center point $\hat{t_0}$. 
Then, we include the higher order chain--rule for completeness. 

\subsection{The function $\frac{\sin t}{t}$}
\label{app:sec:sintpt}

The Taylor series of $\sin t$ 
centered at $t_0 = 0$ is given as 
\begin{equation*}
    \sin t = \sum_{k = 0}^\infty (-1)^k 
        \frac{t^{2k + 1}}{(2k + 1)!}
\end{equation*}
and is convergent for all $t \in \reals$. Consequently, we 
obtain the Taylor series of 
\begin{equation*}
    f(t) := 
        \begin{cases}
            \frac{\sin t}{t}, &\mbox{if } t > 0, \\
            1, &\mbox{if } t = 0\\
        \end{cases}
\end{equation*}
as
\begin{equation}
\label{app:eq:sintpt-taylor}
    f(t) = \sum_{k = 0}^\infty (-1)^k 
        \frac{t^{2k}}{(2k + 1)!}, 
\end{equation}
again, centered at $t_0 = 0$ with the same convergence radius. 
Therefore, 
\begin{equation}
\label{app:eq:sintpt}
    \frac{1}{m!} \deriv{m}{t}{f}(t) = 
    \frac{1}{m!} \sum_{k \geq m / 2}^\infty (-1)^k
        \frac{t^{2k - m}}{(2k - m)! \cdot (2k + 1)}
    \qquad \mbox{for } m = 0, 1, \ldots
\end{equation}
We shall bound these infinite 
series as follows. Let $N \geq m / 2$, then, 
define the finite part $S_f(t; N, m)$ and the 
remainder part $E_f(t; N, m)$ as 
\begin{equation}
\label{app:eq:sintpt-s-e}
\begin{split}
    S_f(t; N, m) &= \frac{1}{m!} 
        \sum_{k \geq m / 2}^N (-1)^k 
            \frac{t^{2k - m}}{(2k - m)! \cdot (2k + 1)} 
        \qquad \mbox{and}\\
    E_f(t; N, m) &= \frac{1}{m!} 
        \sum_{k = N + 1}^\infty 
            (-1)^k \frac{t^{2k - m}}{(2k - m)! \cdot (2k + 1)}.
\end{split}
\end{equation}
The following lemma establishes bounds for the remainder. 
\begin{lemma}
\label{app:lemma:sintpt-remainder}
Let $m, N \in \integers$ with $m \geq 0$ and $N \geq m / 2$. 
Then, 
\begin{equation*}
    E_f(t; N, m) \in 
    \frac{1}{m!} \frac{e^{\, t}}{(2N + 2 - m)!} 
        t^{2N + 2 - m} ~ \cdot ~ [-1, 1] 
\end{equation*}
for all $t \geq 0$. 
\end{lemma}

\begin{proof}
Let $t \geq 0$ and consider 
\begin{equation*}
    | E_f(t; N, m) | \leq 
    \frac{1}{m!} 
        \sum_{k = N + 1}^\infty 
        \frac{t^{2k - m}}{(2k - m)! \cdot (2k + 1)} \leq 
    \frac{1}{m!} 
        \sum_{k = N + 1}^\infty \frac{t^{2k - m}}{(2k - m)!} \leq 
    \frac{1}{m!}
        \sum_{k = 2N + 2 - m}^\infty \frac{t^k}{k!}. 
\end{equation*}
Note that we have obtained the tail of the Taylor series of the 
exponential function centered at $t_0 = 0$. 
The corresponding Lagrange remainder gives us that 
for all integers $K \geq 0$
\begin{equation*}
    \sum_{k = K}^\infty \frac{t^k}{k!} = 
        \frac{1}{K!} \deriv{K}{t}{e^t}(\xi) \cdot t^{K}
\end{equation*}
holds with some $\xi \equiv \xi(K) \in [0, t]$. Observe that 
$\deriv{K}{t}{e^t}(\xi) = e^{\xi}$ 
and that attains its maximum at $\xi = t$ over $\xi \in [0, t]$.
Hence, we obtain 
\begin{equation*}
    \sum_{k = K}^\infty \frac{t^k}{k!} \leq 
        \frac{1}{K!} e^{\, t} \cdot t^K 
        \qquad \mbox{for } t \geq 0. 
\end{equation*}
Finally, setting $K = 2N + 2 - m$ and deriving a bound on 
$E_f(t; N, m)$ from the estimate for 
$|E_f(t; N, m)|$ concludes the proof. 
\end{proof}

Defining 
\begin{equation}
\label{app:eq:sintpt-E-formula}
    \mathbf{E}_f(t; N, m) = 
        \frac{1}{m!} \frac{e^{\, t}}{(2N + 2 - m)!} 
            t^{2N + 2 - m} ~ \cdot ~ [-1, 1]
\end{equation}
together with \eqref{app:eq:sintpt}, \eqref{app:eq:sintpt-s-e}, and 
Lemma~\ref{app:lemma:sintpt-remainder} gives a rigorous computational 
scheme for $f(t)$ and its derivatives, namely, 
\begin{equation*}
\label{app:eq:sintpt-computation}
    \frac{1}{m!} \deriv{m}{t}{f}(t) \in 
        S_f(t; N, m) + \mathbf{E}_f(t; N, m). 
\end{equation*}
We remark that $\lim_{N \to \infty} \mathbf{E}_f(t; N, m) \to \{0\}$ 
for all $t \in \reals$ and $m \geq 0$. 
Figure~\ref{app:fig:sintpt-E} gives an insight on how the obtained 
bound for the remainder behaves. 
\begin{figure}[ht]
    \centering
	\includegraphics[width=0.9\textwidth]{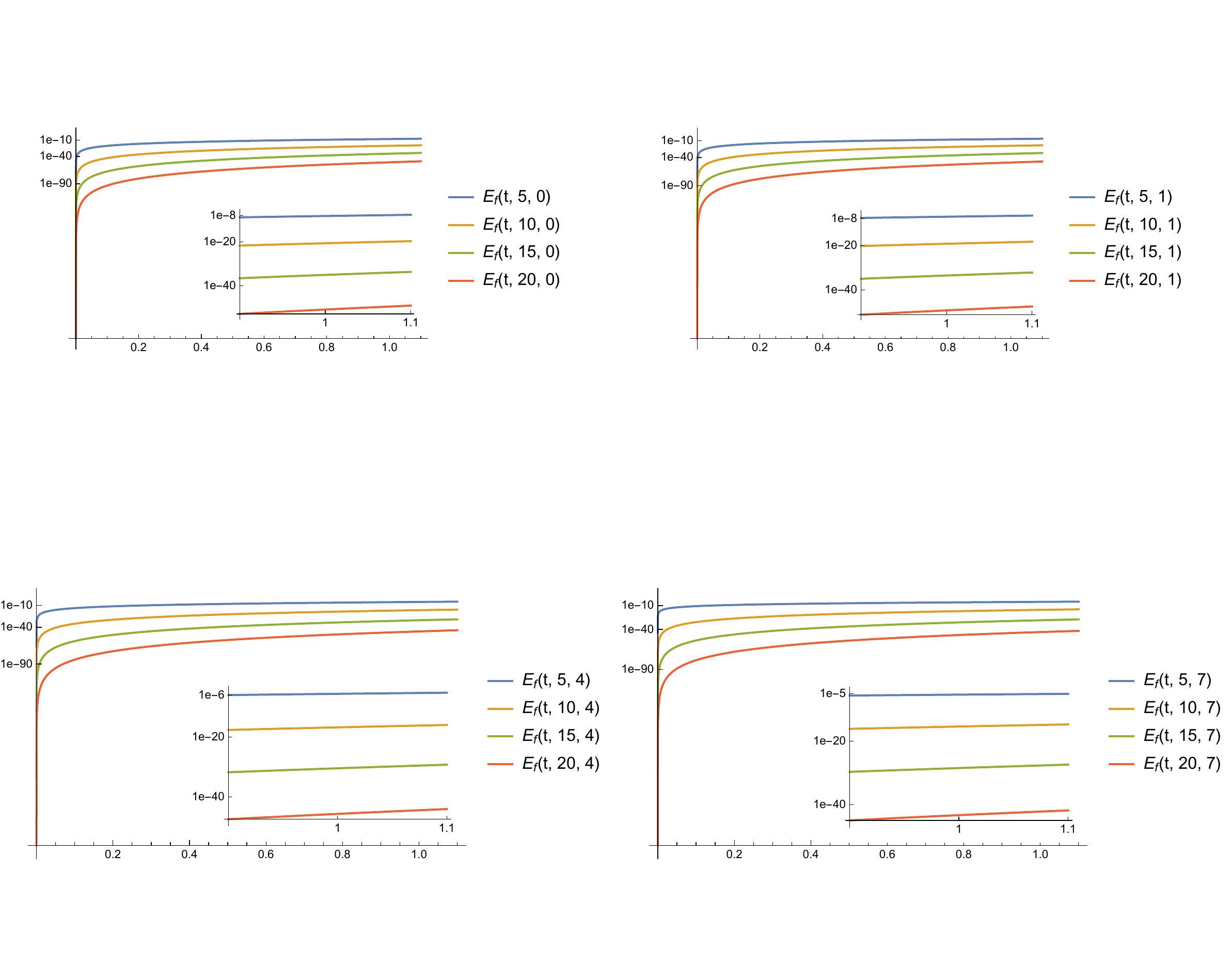}
\caption{
	    \label{app:fig:sintpt-E}
	    The upper bound of $\mathbf{E}_f(t; N, m)$ for various $(N, m)$ over $t \in [0, \tmax]$. 
	}
\end{figure}

\subsection{The function $1 + t^2 F(t) = \frac{\sin t}{t}$}
\label{app:sec:sintpt-taylor-model}
Even though 
there are no issues with directly computing 
$\log \left( f(t) \right)$ using the 
results above, as shown in \eqref{app:eq:xt-scheme}, we will need 
a more sophisticated approach in order to be able to 
tackle the final square root operation 
in the neighbourhood of zero. 
To that end, 
we rewrite expansion \eqref{app:eq:sintpt-taylor} as 
\begin{equation*}
\label{app:eq:sintpt-taylor-model}
    f(t) = 1 + t^2 F(t) = 1 + t^2 
        \sum_{k = 0}^\infty (-1)^{k + 1} 
            \frac{t^{2k}}{(2k + 3)!}, 
\end{equation*}
a \nth{2}--order Taylor model. 
Analogous arguments, as in Section~\ref{app:sec:sintpt}, 
lead to the following. 
\begin{lemma}
\label{app:lemma:sintpt-taylor-model}
\begin{equation*}
    \frac{1}{m!} \deriv{m}{t}{F}(t) \in 
        S_F(t; N, m) + \mathbf{E}_F(t; N, m),  
\end{equation*}
where
\begin{equation*}
\begin{split}
    S_F(t; N, m) &= \frac{1}{m!} 
        \sum_{k \geq m / 2}^N (-1)^{k + 1} 
            \frac{t^{2k - m}}{(2k - m)! \cdot 
                (2k + 1)(2k + 2)(2k + 3)}
    \qquad \mbox{and} \\
    \mathbf{E}_F(t; N, m) &= \mathbf{E}_f(t; N, m). 
\end{split}
\end{equation*}
\end{lemma}

Note that the remainder bound is identical 
to the one in \eqref{app:eq:sintpt-E-formula} as the factor 
in the denominator $(2k + 1)(2k + 2)(2k + 3)$ may be eliminated 
the same way as $(2k + 1)$ in the proof of 
Lemma~\ref{app:lemma:sintpt-remainder}. 

%%%%%%%%%%%%%%%%%%%%%%%%%%%%%%%%%%%%%%%%%%%

\subsection{The function $\frac{\log (1 + t)}{t}$}
\label{app:sec:logtpt}
Following \eqref{app:eq:xt-scheme}, we 
will compute $x(t)$ using the form 
\begin{equation*}
\label{app:eq:xt-practical}
    x(t) = t
        \sqrt{-6 \, F(t) 
            \frac{\log \left( 1 + t^2 F(t) \right)}{t^2 F(t)}}. 
\end{equation*}
Thus, the next step is to analyze  
$g(t) = \tfrac{\log (1 + t)}{t}$, where $t \in (-1, 1)$. This interval 
comes from the well-known expansion of $\log (1 + t)$. 
At $t = 0$, we augment with the limit 
$g(0) := 1$.
Note that the argument of $g(\cdot)$ will be 
$t^2 F(t) = \frac{\sin t}{t} - 1$ that takes values 
roughly in $[-0.189, 0]$.

Let us start from the Taylor series of 
$\log(1 + t)$ centered at $t_0 = 0$, namely,  
\begin{equation*}
    \log(1 + t) = \sum_{k = 1}^\infty (-1)^{k + 1}  
        \frac{t^k}{k} 
\end{equation*}
that is convergent for $|t| < 1$. Then, formally, 
\begin{equation*}
    g(t) = 
        \sum_{k = 0}^\infty (-1)^{k}  
        \frac{t^k}{k + 1} 
\end{equation*}
and 
\begin{equation*}
    \frac{1}{m!} \deriv{m}{t}{g}(t) = 
    \frac{1}{m!} (-1)^m \sum_{k = 0}^\infty (-1)^k
        \frac{t^{k}}{k + m + 1} \, 
        \frac{(k + m)!}{k!}
    \qquad \mbox{for } m = 0, 1, \ldots
\end{equation*}
follow that may be simplified as 
\begin{equation*}
\label{app:eq:logrpr}    
    \frac{1}{m!} \deriv{m}{t}{g}(t) = 
    (-1)^m \sum_{k = 0}^\infty (-1)^k
        \binom{k + m}{m} \frac{t^{k}}{k + m + 1}. 
\end{equation*}
We define 
\begin{equation}
\label{app:eq:logrpr-s-e}
\begin{split}
    S_g(t; N, m) &= (-1)^m 
        \sum_{k = 0}^N (-1)^k
        \binom{k + m}{m} \frac{t^{k}}{k + m + 1}
        \qquad \mbox{and}\\
    E_g(t; N, m) &= (-1)^m 
        \sum_{k = N + 1}^\infty (-1)^k 
            \binom{k + m}{m} \frac{t^{k}}{k + m + 1}  
\end{split}
\end{equation}
for $N \geq 0$ (practically $N \geq m$ so that 
$k + m > 2m$ in the binomial coefficients in $E_g$). We may 
bound the remainder part as detailed below. 
\begin{lemma}
\label{app:lemma:logrpr-taylor}
Let $N \geq m \geq 0$ and $t \in (-1, 1)$. Then, 
\begin{equation*}
    |E_g(t; N, m)| \leq  
    \begin{cases}
        \frac{|t|^{N + 1}}{(1 - |t|)^{N + 2}}, 
            &\mbox{if } m = 0,\\
        \left( \frac{2 e}{m} \right)^m m!
            \binom{m + N + 1}{m} 
            \frac{|t|^{N + 1}}{(1 - |t|)^{m + N + 2}}, 
            &\mbox{else}.\\
    \end{cases}
\end{equation*}
\end{lemma}

\begin{proof}
When $m = 0$, the binomial coefficient $\binom{k + m}{m} = 1$, thus, 
\begin{equation*}
    |E_g(t; N, 0)| \leq  
        \sum_{k = N + 1}^\infty 
            \frac{|t|^{k}}{k + 1} \leq 
            \sum_{k = N + 1}^\infty |t|^k. 
\end{equation*}
On the other hand, for $m > 0$, it is known that
\begin{equation*}
    \binom{k + m}{m} \leq \left( \frac{e \, (k + m)}{m} \right)^m.
\end{equation*}    
Hence, 
\begin{equation*}
\begin{split}
    |E_g(t; N, m)| \leq  
        &\sum_{k = N + 1}^\infty 
            \left( \frac{e \, (k + m)}{m} \right)^m
            \frac{|t|^{k}}{k + m + 1} \leq \\
        &\sum_{k = N + 1}^\infty 
            \left( \frac{e \, (1 + \tfrac{m}{k})}{m} \right)^m
            \frac{k^m |t|^{k}}{k + m + 1} \leq 
        \left( \frac{2 e}{m} \right)^m 
            \sum_{k = N + 1}^\infty k^m |t|^k.
\end{split}
\end{equation*}
Thus, for both cases, it is sufficient to bound the series 
\begin{equation*}
    \sum_{k = N + 1}^\infty k^m |t|^k 
\end{equation*}
for all $m = 0, 1, \ldots$ \,, $N \geq m$ and $|t| < 1$. 

In order to simplify the notation, let $T = |t| \in [0, 1)$ and 
consider the $m$-th derivative of the convergent geometric series 
\begin{equation*}
    \frac{1}{1 - T} = \sum_{k = 0}^\infty T^k 
\end{equation*}
that is
\begin{equation*}
   % \deriv{m}{T}{} 
    \left(\frac{1}{1 - T}\right)^{(m)} = \sum_{k = 0}^\infty 
        \frac{(k + m)!}{k!} T^k. 
\end{equation*}
We may easily bound the remainder of this series starting from 
$k = N + 1$ using, again, the Lagrange formula as
\begin{equation*}
    \sum_{k = N + 1}^\infty 
        \frac{(k + m)!}{k!} T^k = \deriv{m + N + 1}{T}{} \frac{1}{1 - T} 
            \Bigg\vert_{T = \xi} \cdot ~ 
            \frac{1}{(N + 1)!} \cdot T^{N + 1}  
\end{equation*}
with some $\xi \in [0, T]$. The $K$-th derivative of 
$\tfrac{1}{1 - T} = (1 - T)^{-1}$ is given by $K! ~ (1 - T)^{-(K + 1)}$ 
that is clearly maximal for $\xi = T$. Hence, 
\begin{equation*}
    \sum_{k = N + 1}^\infty \frac{(k + m)!}{k!} T^k \leq 
        (m + N + 1)! ~ (1 - T)^{-(m + N + 2)}
        \frac{1}{(N + 1)!} T^{N + 1} 
\end{equation*}
that concludes the proof by noting that
\begin{equation*}
     \sum_{k = N + 1}^\infty k^m T^k \leq 
     \sum_{k = N + 1}^\infty \frac{(k + m)!}{k!} T^k 
\end{equation*}
holds for all $N \geq m \geq 0$ and $T \in [0, 1)$. 
\end{proof}
In summary, letting 
\begin{equation*}
    \mathbf{E}_g(t; N, m) :=  [-1, 1] \cdot 
    \begin{cases}
        \frac{|t|^{N + 1}}{(1 - |t|)^{N + 2}}, 
            &\mbox{if } m = 0,\\
        \left( \frac{2 e}{m} \right)^m m! 
            \binom{m + N + 1}{m} 
            \frac{|t|^{N + 1}}{(1 - |t|)^{m + N + 2}}, 
            &\mbox{else},\\
    \end{cases}
\end{equation*}
provides the computational method 
\begin{equation}
\label{app:eq:logrpr-computation}
    \frac{1}{m!} \deriv{m}{t}{g}(t) \in
        S_g(t; N, m) + \mathbf{E}_g(t; N, m). 
\end{equation}
To analyze the dynamics of \eqref{app:eq:logrpr-computation}, 
observe that the behaviour of the remainder is governed by
\begin{equation*}
   \binom{m + N + 1}{m} \left( \frac{|t|}{1 - |t|} \right)^{N}  
\end{equation*}
for fixed $t \in (-1, 1)$ and $m \geq 0$. Using the same 
bound as above for the binomial, it is easy to see that, 
eventually, 
\begin{equation*}
    N^m \left( \frac{|t|}{1 - |t|} \right)^{N}  
\end{equation*}
determines the limit, when $N \to \infty$. 
Therefore, 
\begin{equation*}
    \lim_{N \to \, \infty} \mathbf{E}_g(t; N, m) = \{ 0 \},
\end{equation*}
when $\tfrac{|t|}{1 - |t|} < 1$ 
that is $|t| < \tfrac{1}{2}$. Recall that for our case this 
will be satisfied as $\tfrac{\sin 1.1}{1.1} - 1 \approx - 0.189$. 
The dynamics of the upper bound of $\mathbf{E}_g(t; N, m)$ is 
demonstrated on Figure~\ref{app:fig:logtpt-E}. 

\begin{figure}[ht]
    \centering
	\includegraphics[width=0.9\textwidth]{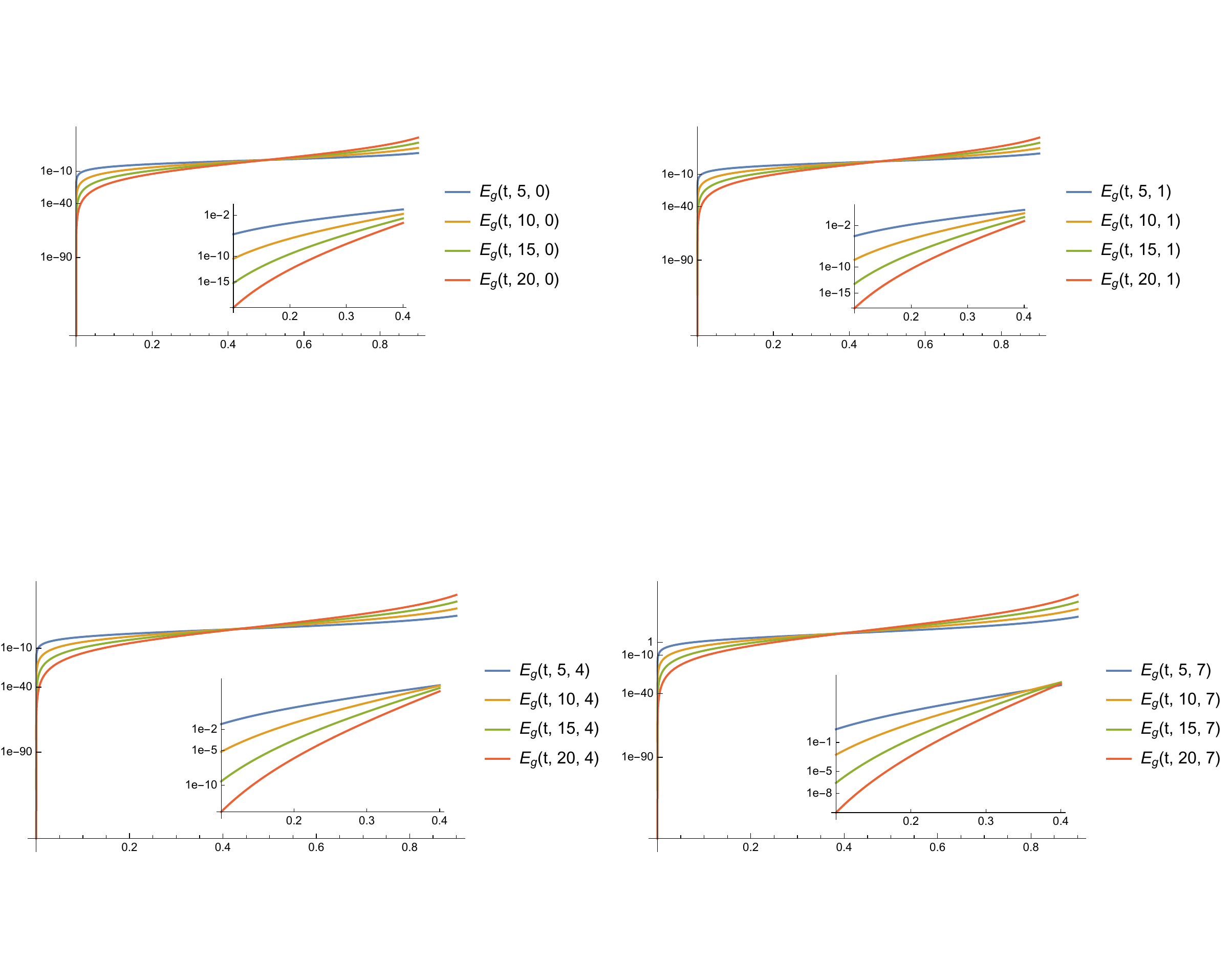}
	\caption{
	    \label{app:fig:logtpt-E}
	    The upper bound of $\mathbf{E}_g(t; N, m)$ for various values. 
	}	

\end{figure}

%%%%%%%%%%%%%%%%%%%%%%%%%%%%%%%%%%%%%%%%%%%

\subsection{The function $\sqrt{t^2 h(t)}$}
\label{app:sec:sqrt}
Assume $t \in [0, T]$ with some $T \geq 0$. 
By itself, the function $\sqrt{t}$ is not differentiable at $t_0 = 0$. 
However, if the argument is of the special form $t^2 h(t)$ 
with $h(t) \neq 0$, then 
the situation changes as 
\begin{equation*}
    \sqrt{t^2 h(t)} = t \sqrt{h(t)}, 
\end{equation*}
hence,  
\begin{equation*}
    \deriv{}{t}{} \sqrt{t^2 h(t)} = 
        \sqrt{h(t)} + t \frac{h'(t)}{2 \sqrt{h(t)}}
\end{equation*}
implying no difficulties for all $t \in [0, T]$. 

%%%%%%%%%%%%%%%%%%%%%%%%%%%%%%%%%%%%%%%%%%%

\subsection{The chain--rule}
\label{app:sec:chain}
There are numerous known formulae for the 
higher order chain rule \cite{app:faa}. We shall use 
the classical one named after Faà di Bruno 
that is written as follows. 

\begin{lemma}[Faà di Bruno]
\label{app:lemma:faa}
Let $f \colon I \to U$ and $g \colon U \to V$ be analytic functions, 
where $I, U, V \subseteq R$ are connected subsets. 
Consider the Taylor expansions 
$f(t) = \sum_{k = 0}^\infty \taylor{f}{k} (t - t_0)^k$ 
centered at $t_0 \in I$ with $t \in I$ and 
$g(x) = \sum_{k = 0}^\infty \taylor{g}{k} (x - x_0)^k$ 
centered at $x_0 = f(t_0)$ for $x \in U$. Then, 
the composite function $(g \circ f)$ attains the Taylor expansion 
$(g \circ f)(t) = \sum_{k = 0}^\infty \taylor{g \circ f}{k} (t - t_0)^k$ 
centered at $t_0$ with the coefficients 
\begin{equation}
\label{app:eq:faa}
\begin{split}
    \taylor{g \circ f}{0} &= \taylor{g}{0}
        \qquad \mbox{and} \\
    \taylor{g \circ f}{k} &= 
        \sum_{\substack{
            b_1 + 2 b_2 + \ldots + k b_k = k \\
            m := b_1 + b_2 + \ldots + b_k
        }} ~ 
        \frac{m!}{b_1! b_2! \ldots b_k!} \taylor{g}{m} 
        \prod_{i = 1}^k 
        \Big( \taylor{f}{i} \Big)^{b_i},
\end{split}
\end{equation}
where $k \geq 1$ and 
$b_1, \ldots, b_k$ are nonnegative integers. 
\end{lemma}

Note that we altered the notation somewhat compared to \cite{app:faa} 
and use Taylor coefficients instead of derivatives, 
this should not cause confusion.

%%%%%%%%%%%%%%%%%%%%%%%%%%%%%%%%%%%%%%%%%%%

\section{Derivatives of $x(t)$}
\label{app:sec:xt}
Using the combination of results of Section~\ref{app:sec:functions}, we 
may attempt to evaluate $x(t)$ and its derivatives based on the steps  
detailed in \eqref{app:eq:xt-scheme}. The expansions of $-6$, $t$, 
and $t^2$ are trivial, so is the application of the product rule; 
for the square root, the computation of Taylor coefficients 
is straightforward \cite{app:rigref, app:rigref2}. 

We used a uniform $N = 20$ when executing our program and imposed 
$0 \not\in \deriv{1}{t}{x}(\tintv)$ 
to hold as an additional requirement 
needed for the inverse computations (that was never violated). 
We have subdivided the original $\tintv$ into smaller intervals 
so that each was no longer than $\approx 0.001$. For each 
of these intervals we attempted to compute the expansion 
of $x(t)$ directly from \eqref{app:eq:xt} as well. This clearly failed 
for those close to $t = 0$, however, whenever it succeeded, 
we compared it with the results from scheme 
\eqref{app:eq:xt-scheme} and used the intersection of the 
two, somewhat independent, results. 

The obtained enclosures are given in Table~\ref{app:table:xt}. 
Each row presents the interval hull of the rigorous bounds 
obtained over all small subintervals. 
In particular, the first one 
establishes the remaining part 
of (1) in Theorem~\ref{app:thm:main}. 

{\renewcommand{\arraystretch}{1.3}
\begin{table}[!htb]
\centering
\begin{tabular}{|l|l|}
\hline
\multicolumn{1}{|l|}{\bf Taylor coefficient} & 
\multicolumn{1}{|l|}{\bf for any $t_0 \in \tintv$ 
    is contained in }\\
\thickhline
$\taylor{x}{0}$ & $[0,1.123840883419833]$ \\
$\taylor{x}{1}$ & $[0.9999999735553898,1.068240487593705]$ \\
$\taylor{x}{2}$ & $[-2.685547075142236\scie05,0.06993582359879109]$ \\
$\taylor{x}{3}$ & $[0.01666661349471042,0.03208155257501275]$ \\
$\taylor{x}{4}$ & $[-9.697809942521872\scie06,0.009448128475336482]$ \\
$\taylor{x}{5}$ & $[0.0009192700177718516,0.003964692939629423]$ \\
$\taylor{x}{6}$ & $[-2.388397092922722\scie06,0.00154475634905294]$ \\
$\taylor{x}{7}$ & $[6.401450846056105\scie05,0.0006646723554643784]$ \\
\hline
\end{tabular}
\vspace{0.5cm}
\caption{{\bf 
    Bounds on Taylor coefficients of $x(t)$ 
    centered at $t_0 \in \tintv$. 
}}
\label{app:table:xt}
\end{table}}

%%%%%%%%%%%%%%%%%%%%%%%%%%%%%%%%%%%%%%%%%%%

\section{Derivatives of $t(x)$}
\label{app:sec:tx}
Now, that we have computed rigorous bounds 
for the Taylor coefficients of $x(t)$ up to 
the desired order for any center $t_0 \in \tintv$, 
we turn our attention to its inverse $t(x)$. 
First, in Section~\ref{app:sec:inverse}, we 
present the general formula for computing the inverse expansion, 
then, we include the results of our computation 
for $t(x)$ in Section~\ref{app:sec:proof}, thereby concluding 
the proof of (3) in Theorem~\ref{app:thm:main}. 

%%%%%%%%%%%%%%%%%%%%%%%%%%%%%%%%%%%%%%%%%%%

\subsection{Derivatives of the inverse function}
\label{app:sec:inverse}
Practical formulae for Taylor expansion of the inverse function 
based on the coefficients of the original one 
are rather scarce. For our purposes it is reasonable to 
utilize the result of Faà di Bruno, seen in Section~\ref{app:sec:chain}, 
directly. 

Assume that $x(t)$ has the expansion 
$x(t) = \sum_{k = 0}^\infty \taylor{x}{k} (t - t_0)^k$ centered at $t_0$ and 
$\taylor{x}{1} \neq 0$. Then, for the inverse we may construct the expansion 
$t(x) = \sum_{k = 0}^\infty \taylor{t}{k} (x - x_0)^k$
centered at $x_0 = x(t_0)$ as  
\begin{equation}
\label{app:eq:inverse-rule} 
\begin{split}
    \taylor{t}{0} &= t_0, \\
    \taylor{t}{1} &= \frac{1}{\taylor{x}{1}}, 
        \qquad \mbox{and} \\
    \taylor{t}{k} &= 
        - \sum_{\substack{
            b_1 + 2 b_2 + \ldots + k b_k = k \\
            m := b_1 + b_2 + \ldots + b_k \\
            m \neq k
        }} ~ 
        \frac{m!}{b_1! b_2! \ldots b_k!} \taylor{t}{m} 
        \Big( \taylor{x}{1} \Big)^{b_1 - k} ~
        \prod_{i = 2}^k 
        \Big( \taylor{x}{i} \Big)^{b_i},  
\end{split}
\end{equation}
for $k \geq 2$. 
The first two coefficients are trivial. The general part is 
a consequence of Lemma~\ref{app:lemma:faa} applied to 
$(t \circ x)(t)$ by observing that 
$\taylor{t \circ x}{k} = 0$ for $k \geq 2$ and that
in the sum the only term 
containing $\taylor{t}{k}$ 
(that is $\taylor{g}{k}$ in the original Lemma) 
is given by $b_1 = k$ and $b_i = 0$ for all other $i$-s as
\begin{equation*}
    \frac{k!}{k! \, 1! \, \ldots 1!} \taylor{t}{k} 
    \Big( \taylor{x}{1} \Big)^{k} = \taylor{t}{k} 
    \Big( \taylor{x}{1} \Big)^{k}. 
\end{equation*}

%%%%%%%%%%%%%%%%%%%%%%%%%%%%%%%%%%%%%%%%%%%

\subsection{Proof of (3) in Theorem~\ref{app:thm:main}}
\label{app:sec:proof}

We have applied \eqref{app:eq:inverse-rule} on  
each of the subintervals and  
the corresponding expansion of $x(t)$, see Section~\ref{app:sec:xt}. 
The interval hull of the results is presented in Table~\ref{app:table:tx}. 
Using that 
$\taylor{t}{\goalOrder} = 
    \tfrac{1}{\goalOrder!} 
    \deriv{\goalOrder}{x}{t}\left( x_0 \right)$, 
we directly obtain the claim of (3) in 
Theorem~\ref{app:thm:main}. 

{\renewcommand{\arraystretch}{1.3}
\begin{table}[!h]
\centering
\begin{tabular}{|l|l|}
\hline
\multicolumn{1}{|l|}{\bf Taylor coefficient} & 
\multicolumn{1}{|l|}{\bf for any $x_0 \in \xintv$ 
    is contained in }\\
\thickhline
$\taylor{t}{0}$ & $[0,1.1]$ \\
$\taylor{t}{1}$ & $[0.9361187968568556,1.000000026444611]$ \\
$\taylor{t}{2}$ & $[-0.05741807585325204,2.685547249459667\scie05]$ \\
$\taylor{t}{3}$ & $[-0.01769296208858369,-0.01666567358306551]$ \\
$\taylor{t}{4}$ & $[-0.0004154319065142972,1.41737288630298\scie05]$ \\
$\taylor{t}{5}$ & $[-8.896646385491572\scie05,0.0001353167787555311]$ \\
$\taylor{t}{6}$ & $[-4.791003638747524\scie05,0.0001399782086031966]$ \\
$\taylor{t}{7}$ & $[-7.620438955153176\scie05,9.988885937812383\scie05]$ \\
\hline
\end{tabular}
\vspace{0.5cm}
\caption{{\bf 
    Bounds on Taylor coefficients of $t(x)$ 
    centered at $x_0 \in \xintv$. 
}}
\label{app:table:tx}
\end{table}}

%%%%%%%%%%%%%%%%%%%%%%%%%%%%%%%%%%%%%%%%%%%

\end{document}